\documentclass[11pt]{article}

\pdfoutput=1

\usepackage{amsfonts,amssymb,graphicx}
\usepackage{amsfonts}
\usepackage{amsmath, amscd}
\usepackage{amstext}
\usepackage{amsthm}
\usepackage{xspace}

\usepackage{fullpage}
\usepackage{subfig}
\usepackage{sidecap}

\newcommand{\Z}{\mathbb Z}
\newcommand{\N}{\mathbb N}

\newcommand{\Q}{\mathbb Q}

\newcommand{\EE}{\mathcal{E}}
\newcommand{\PP}{\mathcal{P}}

\newtheorem{theorem}{Theorem}[section]
\newtheorem{corollary}[theorem]{Corollary}
\newtheorem{lemma}[theorem]{Lemma}
\newtheorem{proposition}[theorem]{Proposition}
\newtheorem{definition}[theorem]{Definition}
\newtheorem{remark}[theorem]{Remark}
\def\be{\begin{equation}}
\def\ee{\end{equation}}

\begin{document}
\title{Dynamics of continued fractions and kneading sequences of unimodal maps}
\author{ 
C. Bonanno\thanks{Dipartimento di Matematica Applicata,
Universit\`a di Pisa, via F. Buonarroti 1/c, I-56127 Pisa, Italy,
email: $<$bonanno@mail.dm.unipi.it$>$}\and
C. Carminati\thanks{Dipartimento di Matematica,
Universit\`a di Pisa, Largo Bruno Pontecorvo 5, I-56127, Italy,
email: $<$carminat@dm.unipi.it$>$}\and
S. Isola
\thanks{Dipartimento di Matematica e Informatica, Universit\`a
di Camerino, via Madonna delle Carceri, I-62032 Camerino, Italy.
e-mail: $<$stefano.isola@unicam.it$>$}\and
G. Tiozzo\thanks{Department of Mathematics, Harvard University, One Oxford Street Cambridge MA 02138 USA, e-mail: $<$tiozzo@math.harvard.edu$>$}
}

\maketitle

\begin{abstract}
In this paper we construct a correspondence between the parameter spaces of two families of one-dimensional dynamical systems,
 the $\alpha$-continued fraction transformations $T_\alpha$ and unimodal maps. This correspondence identifies 
bifurcation parameters in the two families, and allows one to transfer topological and metric properties
from one setting to the other. As an application, we recover results about the 
real slice of the Mandelbrot set, and the set of univoque numbers.
\end{abstract}

\section{Introduction}

The goal of this paper is to discuss an unexpected connection between
the parameter spaces of two families of one-dimensional dynamical
systems, and establish an explicit correspondence between the bifurcation
parameters for these families.

The family $(T_\alpha)_{\alpha \in (0,1]}$ of $\alpha$-continued
  fraction transformations, defined in \cite{Na}, is a family of discontinuous 
interval maps, which generalize the well-known Gauss map. For each $\alpha \in (0,1]$, the map $T_\alpha$
from the interval $[\alpha-1, \alpha]$ to itself is defined as $T_\alpha(0) = 0$ and, for $x \neq 0$,  
$$T_\alpha(x) := \frac{1}{|x|} - c_{\alpha, x}$$
where $c_{\alpha, x} = \left\lfloor \frac{1}{|x|} + 1 - \alpha \right\rfloor$ is a positive integer.
For every $\alpha \in [0,1]$ at most two of the infinitely many
branches of $T_\alpha$ fail to be surjective, hence $T_\alpha$ has
infinite topological entropy. On the other hand, for all $\alpha>0$
the map $T_\alpha$ admits a unique invariant probability measure
$\mu_\alpha$ which is absolutely continuous with respect to Lebesgue
measure, and one can consider the metric entropy $h=h(T_\alpha,
\mu_\alpha)$ with respect to this measure; $h$ can be
computed by Rohlin's formula  and is finite (see \cite{LM}).

Several authors have studied the variation of $h$
as a function of the parameter $\alpha$. 
Numerical evidence 
that the entropy $h$ is continuous but non-monotone was  first produced in \cite{LM}. 
Subsequently, \cite{NN} found out that the entropy is monotone over intervals in parameter space 
for which the orbits of the two endpoints collide after a finite number of steps (see Equation \eqref{matching} on page \pageref{matching}).
This analysis is completed in \cite{CT}, where intervals of parameters 
for which such relation holds are classified, and it is proven that 
the union of all such intervals has full measure.
The complementary set, denoted by $\EE$, is the set of parameters across which 
the combinatorics of $T_\alpha$ changes, hence it will  
be called the \emph{bifurcation set}.





The second object  we consider is the family of \emph{unimodal maps}, i.e. smooth maps of the interval with only one critical point,
the most famous example being the logistic family.
To any such map one can associate a \emph{kneading invariant} \cite{MT} which encodes the dynamics of the critical point
and determines the combinatorial type of the map. Using an appropriate coding (see \cite{IP}, \cite{Is}) the set of all 
kneading invariants  which arise from unimodal maps can be represented as the points of a set $\Lambda$
defined in terms of the \emph{tent map} $T(x):=\min\{2x, 2(1-x)\}$ 

\begin{equation} \label{lambdachar}
\Lambda := \{ x \in [0, 1] : T^k(x) \leq x \ \forall k \in \mathbb{N} \}
\end{equation}

The set $\Lambda$ is  uncountable, totally disconnected and closed; the topology of $\Lambda$ reflects the variation 
in the dynamics of the corresponding maps as the parameter varies: for instance families of isolated points in $\Lambda$ 
correspond to period doubling cascades (see Section \ref{pedou}).
This set also parametrizes the real slice of the boundary of the Mandelbrot set $\mathcal{M}$ (see Section \ref{mandel}), 
because for each admissible kneading sequence there exists exactly one real parameter on the boundary with
 that kneading sequence. 

The key result of this paper (Section \ref{main}) is the following correspondence between the bifurcation 
parameters of the two families:

\begin{theorem} \label{mth} The sets $\Lambda \setminus \{0\}$ and $\mathcal{E}$ are homeomorphic. More precisely, the map
$\varphi: [0, 1] \rightarrow [\frac{1}{2},1]$ given by 
$$ x = {1\over \displaystyle a_1 + {1\over
\displaystyle a_2 + {1\over\displaystyle a_3 +{1\over\ddots }}}}
\mapsto \varphi(x) = 0.\underbrace{11\ldots1}_{a_1}\underbrace{00\ldots0}_{a_2}\underbrace{11\ldots1}_{a_3}\ldots $$
is an orientation-reversing homeomorphism which takes $\mathcal{E}$ onto $\Lambda\setminus\{0\}$.
\end{theorem}

As a consequence, intervals in the parameter space of $\alpha$-continued fractions where a matching 
between the orbits of the endpoints occurs are in one-to-one correspondence with real hyperbolic components of the Mandelbrot set, 
i.e. intervals in the parameter space of real quadratic polynomials where the orbit of the critical point is attracted to a 
periodic cycle. In terms of entropy, intervals over which the metric entropy of $\alpha$-continued fractions is monotone 
(and conjecturally smooth) are mapped to parameter intervals in the space of quadratic polynomials where the topological 
entropy is constant (see figure \ref{NakadavsMandelbrot}). 
For instance, the matching interval $([0; \overline{3}], [0; \overline{2, 1}])$, identified
in \cite{LM} and \cite{NN}, corresponds to the ``airplane component'' of period $3$ in the Mandelbrot set. 


This dictionary illuminates many connections between seemingly unrelated 
objects in real dynamics, complex dynamics, and arithmetic, which we will explore in the rest of the paper. 
For example, the following properties of $\Lambda$ follow immediately by using the combinatorial tools developed in \cite{CT}:




\begin{enumerate}
\item[(i)] the set $\Lambda$ can be constructed via a bisection algorithm (Section \ref{bisec})
\item[(ii)] the sequences of isolated points in $\mathcal{E}$ observed experimentally in (\cite{CMPT}, Section 4.2) are the images 
of the well-known period doubling cascades for unimodal maps (Section \ref{pedou})
\item[(iii)] the derived set $\Lambda'$ is a Cantor set (Section \ref{top})
\item[(iv)] the Hausdorff dimension of $\Lambda$ is $1$ (Section \ref{HD})
\end{enumerate}

\begin{SCfigure} \label{NakadavsMandelbrot}
\includegraphics[scale=0.2]{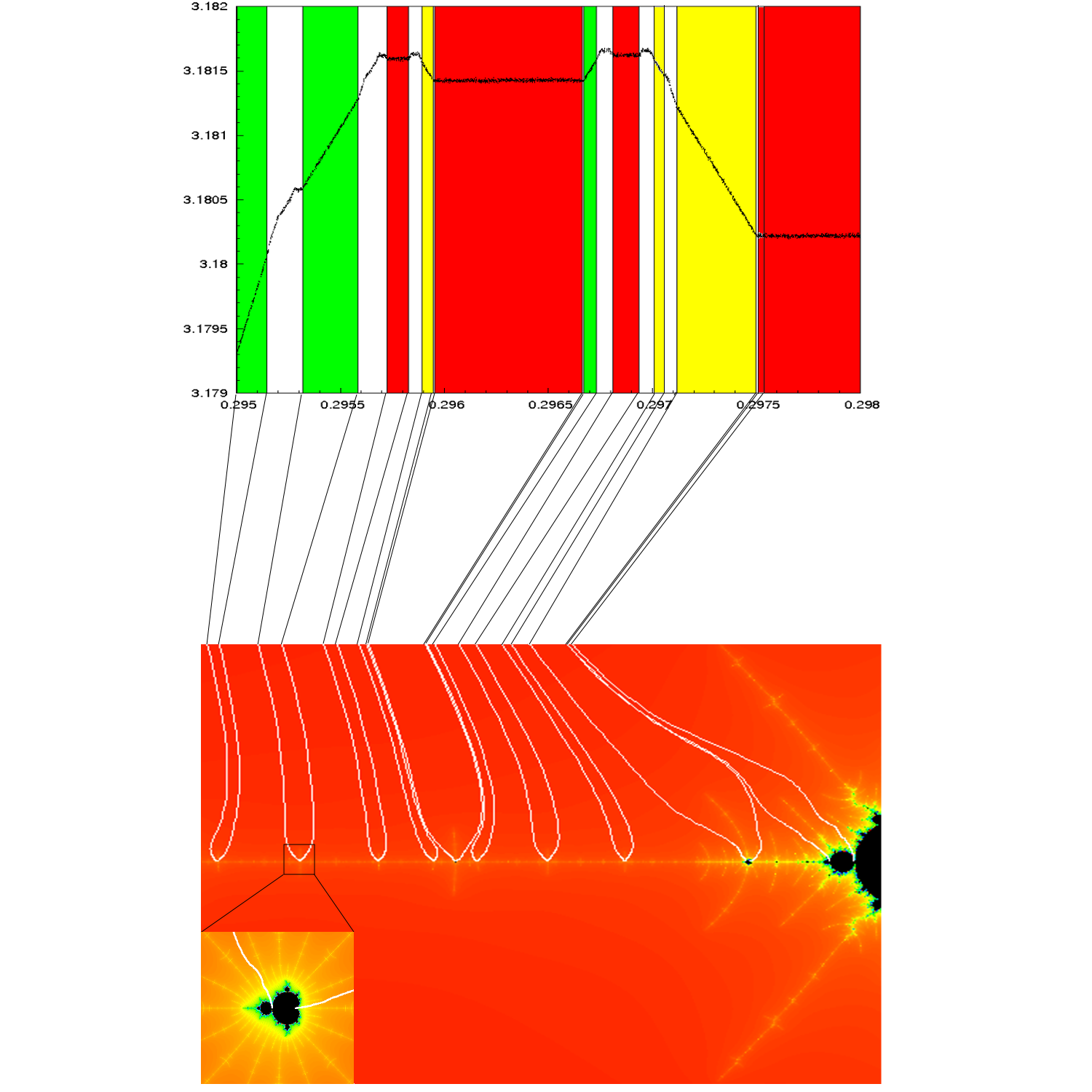}

\caption{Correspondence between the parameter space of $\alpha$-continued fraction transformations and 
the Mandelbrot set. On the top: the entropy of $\alpha$-c.f. as a function of $\alpha$, from \cite{LM}; 
colored strips correspond to \emph{matching intervals}. At the bottom: a section of the Mandelbrot set 
along the real line, with external rays landing on the real axis. 
Matching intervals on the top figure correspond to hyperbolic components on the bottom.} 
\end{SCfigure}


From the above properties of $\Lambda$ we derive some consequences on corresponding sets.
 For example, (iv) implies that the set of external rays of the Mandelbrot set which land on
 the real axis has full Hausdorff dimension (Section \ref{mandel}), a result first obtained in \cite{Za}.


Moreover, in Section \ref{univ} we give an arithmetic
interpretation of $\Lambda$. Namely, aperiodic elements of
$\Lambda$ correspond to binary expansions of \emph{univoque numbers},
i.e. the numbers $q \in (1, 2)$ such that $1$ admits a unique
representation in base $q$. Univoque numbers have been studied by
Erd\"os, Horv\'ath and Jo\'o \cite{EHJ} and many others
(see Section \ref{univ} for references), and once again the translation of
statements (i)-(iv) immediately yields several results previously
obtained by different authors.


\vskip 0.3 cm

Moreover, our new characterization of the set $\EE$ (see Lemma
\ref{characterizeE}) shows that it  is essentially the same object
appearing in a paper of Cassaigne as the spectrum of recurrence
quotients for cutting sequences of geodesics on the torus (\cite{Ca},
Theorem 1.1). Indeed, our results on $\EE$ answer some questions
raised in \cite{Ca}, such as the computation of Hausdorff
dimension. 


Finally, we would like to emphasize that the correspondence 
of Theorem \ref{mth} opens up many questions, and it is especially natural to ask to what
 extent results in the well-developed theory of unimodal maps can be translated into the continued fraction setting.
For instance, it would be interesting to identify an analogue of renormalization for the $\alpha$-continued fractions, 
and to further explore the relation between the entropy of that family (\cite{LM}, \cite{CMPT}, \cite{Ti}, \cite{KSS}) 
and the topological entropy of unimodal maps (\cite{MT} and \cite{Do} among others).

\section*{Acknowledgements}
 We wish to thank J-P. Allouche and H. Bruin for their helpful comments. G.T. wishes to thank C. McMullen 
for many useful discussions. C.B. is sponsored by project MIUR - PRIN2009 ``Variational and topological methods in the study of nonlinear phenomena", Italy, and by the ``Distinguished Scientist Fellowship Program (DSFP)", King Saud University, Riyadh, Saudi Arabia.
The lower half of Figure 1 has been generated using the software Mandel by W. Jung.







\section{Preliminaries} \label{prelim}

\noindent
Let $T,F, G$ denote the \emph{tent map}, the \emph{Farey map} and the \emph{Gauss map}  of $[0,1]$, given by\footnote{Here $
\lfloor x \rfloor$ and $\{x\}$ denote the integer and the fractional part of $x$, respectively, so that $x=\lfloor x \rfloor+\{x\}$.
}
$$
 T(x):=\left\{
\begin{array}{cl}
\displaystyle 2x  & {\rm if}\ 0\leq x< \frac{1}{2} \\[0.5cm]
 \displaystyle 2(1-x) & {\rm if}\ \frac{1}{2} \leq  x\leq 1
\end{array} \right.  
\ \ \ \ 
  F(x):=\left\{
\begin{array}{cl}
\displaystyle \frac{x}{1-x}  & {\rm if}\ 0\leq x< \frac{1}{2} \\[0.5cm]
\displaystyle \frac{1-x}{x}& {\rm if}\ \frac{1}{2} \leq  x\leq 1
\end{array} \right.
$$  
 and $G(0):=0$, $G(x) := \left\{ \frac 1 x \right\}, \quad x\ne 0$.
The action of $F$ and $T$ can be nicely illustrated with different symbolic codings of numbers. Given $x\in [0,1]$ we can expand it in (at least) two ways: using a \emph{continued fraction expansion}, i.e.
$$
 x = {1\over \displaystyle a_1 + {1\over
\displaystyle a_2 + {1\over\displaystyle a_3 +{1\over\ddots }}}}\equiv
[0; a_1,a_2,a_3,\dots]\quad , \quad a_i\in \N
$$
and a \emph{binary expansion}, i.e.
$$
x= \sum_{i\geq 1} b_i\, 2^{-i} \equiv 0. b_1\, b_2 \dots \quad , \quad b_i \in \{0,1\}
$$
The action of $T$ on binary expansions is as follows: for $\omega \in \{0,1\}^{\N}$,
\be\label{tenda}
T(0.\, 0\,\omega ) = 0.\, \omega \quad , \quad  T(0.\,1\, \omega ) = 0.\, {\hat \omega}
\ee
where ${\hat \omega}={\hat \omega}_1{\hat \omega}_2\dots$ and
$\hat{0} =1$, $\hat{1}=0$. 
The actions of $F$ and $G$ are given by
$
F([0;a_1,a_2,a_3,\dots] )=[0; a_1-1,a_2,a_3,\dots] 
$
if $a_1>1$, while $F([0; 1,a_2,a_3,...])=[0; a_2,a_3,...]$, and
$G([0;a_1,a_2,a_3,\dots] )=[0; a_2,a_3,\dots]$.
 As a matter of fact $G$ can be obtained by $F$ by inducing on the interval $I_1=[1/2, 1]$, i.e. 
\be
G(x) = F^{\lfloor 1/x \rfloor}(x)\quad , \quad x\ne 0
\ee

\noindent
Now, given 
$x=[0;a_1,a_2,a_3,\dots]$, one may ask what is the number obtained by
interpreting the partial quotients $a_i$ as the lengths of successive blocks in  the
dyadic expansion of a real number in $[0,1]$; this defines \emph{Minkowski's question mark} function
$?: \ [0,1] \rightarrow [0,1]$ 
\be ?(x)=\sum_{k\geq 1}(-1)^{k-1}\, 2^{-(a_1+\cdots
+a_k-1)} = 
0.\, {\underbrace
{00\dots 0}_{a_1-1}}\;{\underbrace {11\dots 1}_{a_2}}\;{\underbrace
{00\dots 0}_{a_3}}\; \cdots 
\ee
\begin{figure}[h]
  \centering
\subfloat[Tent map]{\label{fig:gull}\includegraphics[width=0.3\textwidth]{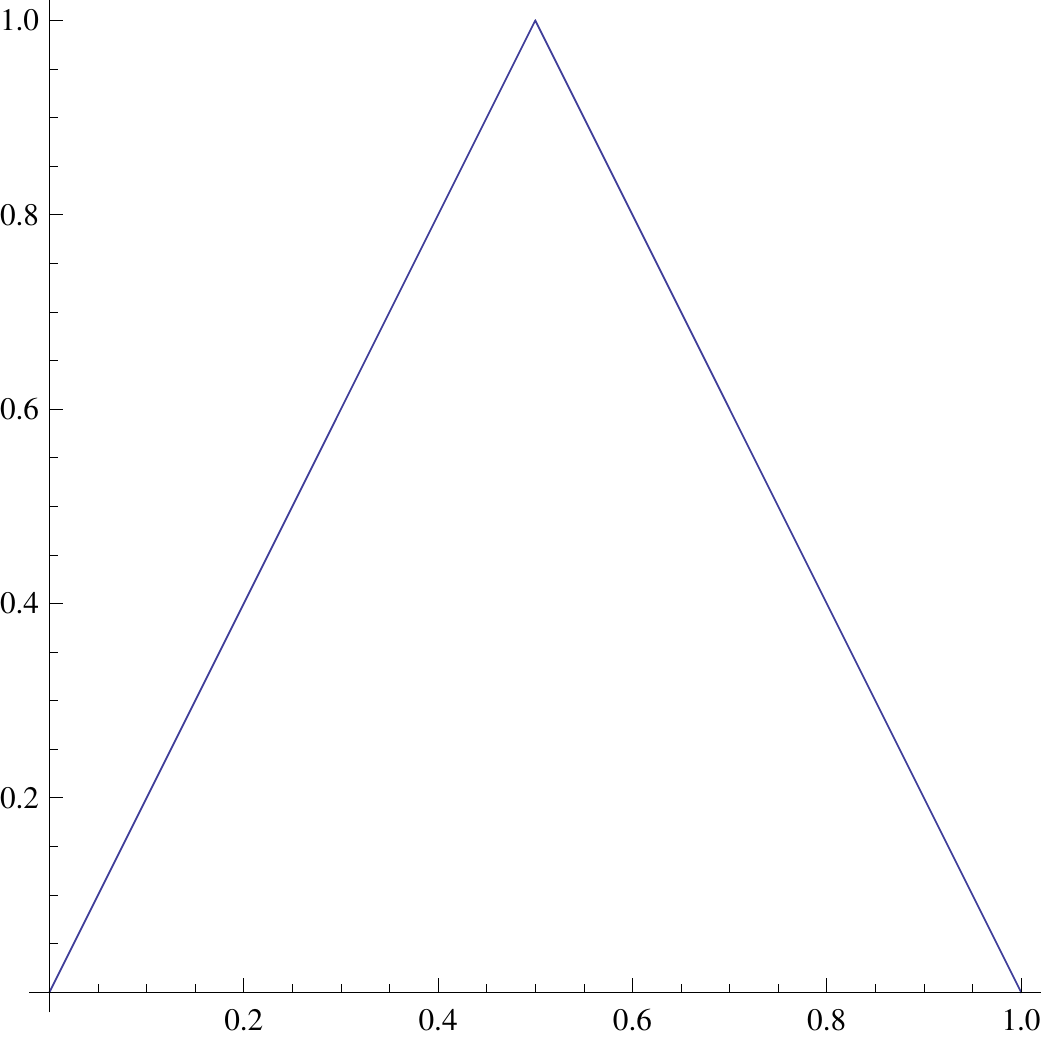}}       
\quad
  \subfloat[Minkowski map]{\label{fig:mouse}\includegraphics[width=0.3\textwidth]{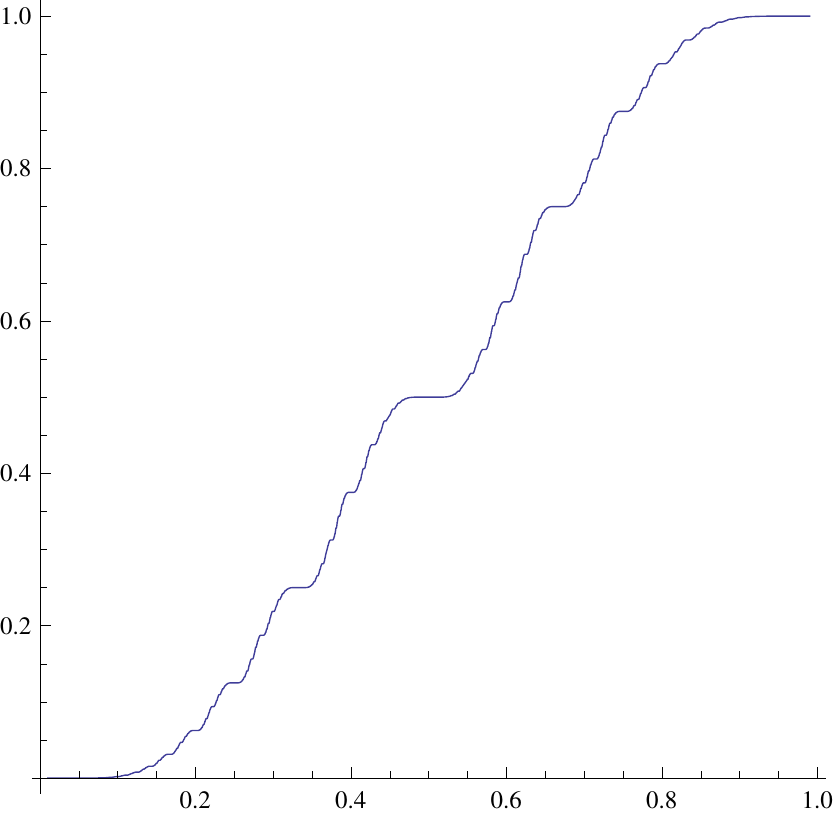}}
\quad
  \subfloat[Farey map]{\label{fig:tiger}\includegraphics[width=0.3\textwidth]{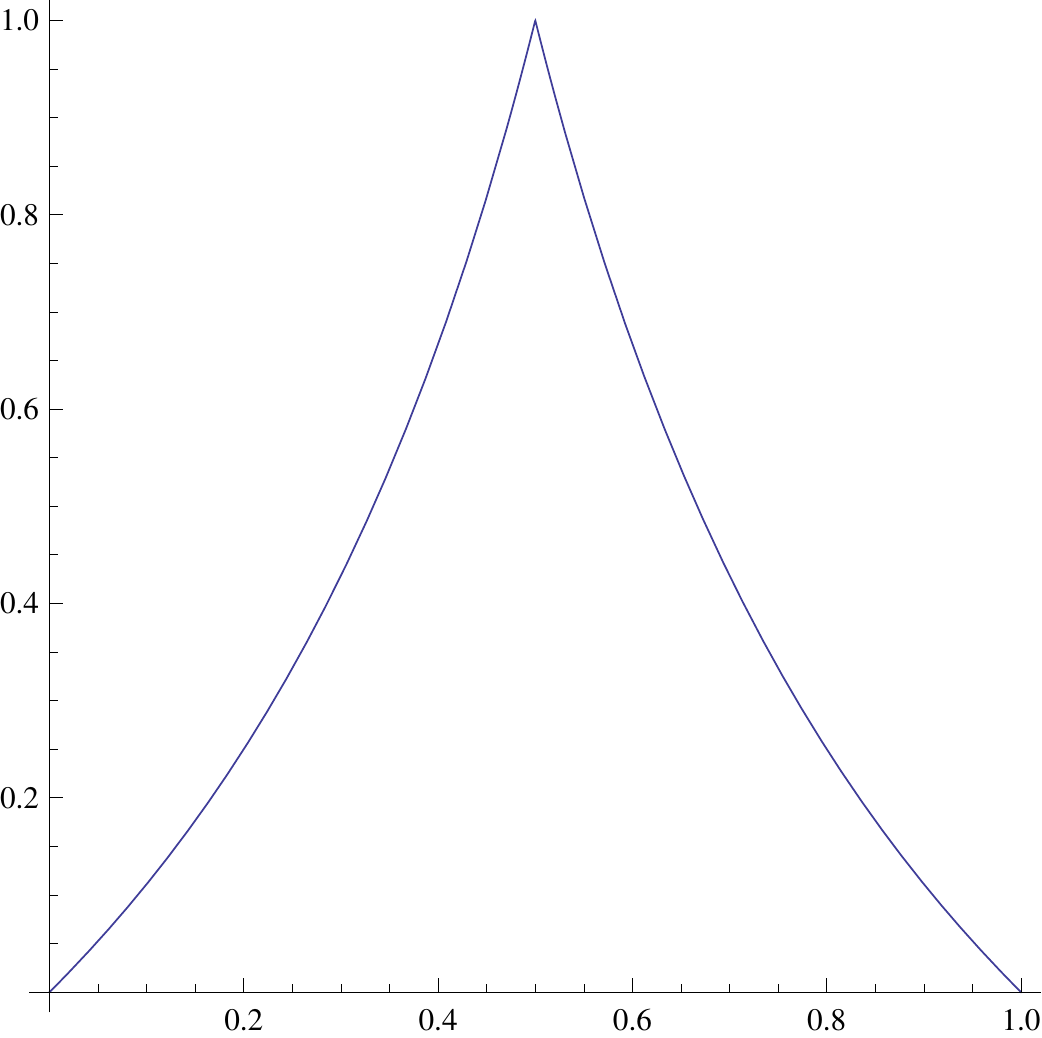}}
\end{figure}
The questionmark function $?(x)$ has the following properties (see \cite{Sa}):
\begin{itemize}
\item it is strictly increasing from $0$ to $1$  and H\"older continuous of exponent $\beta =\frac{\log
2}{2 \log \frac{\sqrt{5}+1}{2}}$;
\item $x$ is rational iff $?(x)$ is of the form $k/2^s$, with $k$ and
  $s$ integers;
\item $x$ is a quadratic irrational iff $?(x)$ is a (non-dyadic)
rational;
\item $?(x)$ is a singular function: its derivative vanishes Lebesgue-almost everywhere;
\item it satisfies the functional equation $?(x)+?(1-x)=1$.
\end{itemize}
We shall see later that the question mark function $?$ plays a key role in the main result of this paper (Theorem \ref{mth})
because it conjugates the actions of $F$ and $T$. For a more general analysis of the role played by $?$ in a dynamical setting see \cite{BI}.

\section{The bifurcation sets} \label{exset}

\subsection{The bifurcation set for continued fraction transformations} \label{E}

Let $r \in (0, 1) \cap \mathbb{Q}$ be a rational number, and let $r = [0; a_1, \dots, a_n]$ be its continued fraction expansion, with
$a_n \geq 2$. The \emph{quadratic interval} associated to $r$ will be the 
open
interval $I_r$ with endpoints
$$[0; \overline{a_1, \dots, a_{n-1}, a_n}] \quad \textup{and} \quad [0; \overline{a_1, \dots, a_{n-1}, a_n - 1, 1}]$$
where $[0; \overline{a_1, \dots, a_n}]$ denotes the real number 
 with periodic continued fraction expansion of period 
$(a_1,\dots,a_n)$ .
The number $r$ is called the \emph{pseudocenter} of $I_r$. We also define the degenerate quadratic interval $I_1:=(g,1]$, where $g:=[0;\bar 1]$ is the golden number.

\begin{definition} The \emph{bifurcation set} for continued fractions is  
\begin{equation} \label{Edef}
{\cal E} := [0,1]\setminus \bigcup_{r \in (0, 1]\cap \mathbb{Q}} I_r
\end{equation}
\end{definition}
 Quadratic intervals which are not properly contained in any larger
 quadratic interval are called \emph{maximal} quadratic intervals. By
 (\cite{CT}, Section 2.2), if two quadratic intervals are not disjoint,
 they are both contained in a maximal quadratic interval; thus, the
 connected components of the complement of $\EE$ are precisely the
 maximal intervals.
For all parameters belonging to a maximal quadratic interval, the $\alpha$-continued fraction transformation $T_\alpha$ 
satisfies a \emph{matching condition} between the orbits of the endpoints, with fixed combinatorics:

\begin{theorem}[\cite{CT}, Theorem 3.1] 
Let $I_r$ be a maximal quadratic interval, and $r = [0; a_1, \dots, a_n]$ with $n$ even. Let 
$$N = \sum_{i \textup{ even}} a_i \qquad M = \sum_{i \textup{ odd}} a_i$$
Then for all $\alpha \in I_r$, 
\begin{equation} \label{matching}
T_\alpha^{N+1}(\alpha) = T_\alpha^{M+1}(\alpha-1) \qquad \forall \alpha \in I_r
\end{equation}
\end{theorem}

As a consequence, by \cite{NN}, the metric entropy $h(T_\alpha)$ is locally monotone on the complement of $\EE$.
The set $\mathcal{E}$ can be given the following new characterization:



\begin{lemma} \label{characterizeE}
Let $x\in [0,1]$. The following are equivalent:
\begin{enumerate}
\item[(a)] $x\in {\cal E}$
\item[(b)] $G^k(x)\geq x$, $\forall k\in \N$
\item[(c)] $F^k(x)\geq x$, $\forall k\in \N$
\item[(d)] $F^k\circ \Psi_1 (x) \leq \Psi_1 (x)$\, $\forall k \in \N$,
where $\Psi_1(x) := 1/(1+x)$ is an inverse branch of $F$.
\end{enumerate}
\end{lemma}
\noindent
{\sl Proof.} 
Let $x = [0; a_1, a_2, \dots]$ be the continued fraction expansion of $x$.

\noindent $(b) \Rightarrow (c)$. Fix $k \geq 1$, so 
$F^k(x) = [0; h, a_{n+1}, a_{n+2}, \dots]$
for some $n$ and $1 \leq h \leq a_n$. Hence
$$F^k(x) \geq [0; a_n, a_{n+1}, \dots] = G^{n-1}(x) \geq x $$

\noindent $(c) \Rightarrow (d)$. Let $J\subset \N$ be the subsequence given by the partial quotients sums, 
i.e.  $J=\{\sum_{i=1}^n a_i\}_{n\geq 1}$, and note that $y := \Psi_1(x) = [0; 1, a_{1},  a_{2}, \dots ]$.
If $k\notin J$ then $F^k(y) \leq \frac{1}{2} \leq y$. Otherwise, $F^k(y)=[0; 1, a_{n+1}, a_{n+2}, \dots ]
=\Psi_1(F^k(x))$ for some $n\geq 1$ so that
$$ x=[0; a_1, a_2, a_3,...] \leq [0;  a_{n+1},
 a_{n+2}, \dots ]=F^k(x) \Rightarrow F^{k}(y) \leq y $$

\noindent $(d) \Rightarrow (b)$. 
Fix $k \geq 1$, and let $s := \sum_{j = 1}^k a_j$, so that $F^s(x) = G^k(x)$. Now, 
$$
F^s(\Psi_1(x)) = [0; 1, a_{k+1}, a_{k+2}, \dots] \leq [0; 1, a_1, a_2, \dots] = \Psi_1(x)
$$
therefore $[0; a_{k+1}, \dots ] = G^k(x) \geq x = [0; a_1, \dots]$.
 
\noindent $(a) \Rightarrow (c)$. 
Let $x \in \mathcal{E}$ be the endpoint of a maximal interval, hence 
$x = [0; \overline{a_1, \dots, a_n}]$.
Once you fix $k \geq 0$, then $G^k(x) = [0; \overline{a_{l+1}, \dots, a_n, a_1, \dots, a_l}]$ 
with $l \equiv k \mod n$. Now, by \cite{CT}, Proposition 4.5  
$$G^k(x) = [0; \overline{a_{l+1}, \dots, a_n, a_1, \dots, a_l}] \geq [0; \overline{a_1, \dots, a_l, a_{l+1}, \dots, a_n}] = x$$
i.e. $(b)$, and hence $(c)$.
Since 
%
$\mathcal{E}$ has empty interior
 (\cite{CT}, Theorem 1.2), every point in $\mathcal{E}$ is a limit point of a 
sequence of endpoints of maximal intervals, so inequality $(c)$ extends to all of $\mathcal{E}$ 
by continuity of $F$.

\noindent $(b)\Rightarrow (a)$. If $x\notin \mathcal{E}$, then $x$ belongs to a maximal interval $I_a = (\alpha^-, \alpha^+)$.
Let $a = [0; A^-] = [0; A^+]$ be the two continued fraction expansions of the rational number $a$, 
in such a way that $\alpha^- = [0; \overline{A^-}]$ and $\alpha^+ = [0; \overline{A^+}]$. Here $A^-$ is a finite string 
of integers of odd length 
$l$, and $A^+$ is a finite string of integers of even length 
$m$.
The sequence 
$$a = [0; A^+] < [0; A^+A^+]< \dots < [0; (A^+)^n] < [0; (A^+)^{n+1}] < \dots $$
tends to $[0; \overline{A^+}] = \alpha^+$, hence 
for any $x \in [a, \alpha^+)$ there exists $n \geq 0$ such that 
$$[0; (A^+)^{n+1}] \leq x < [0; (A^+)^{n+2}]$$ 
therefore $y := G^{m}(x) < [0; (A^+)^{n+1}] \leq x$. 
If instead $x \in (\alpha^-, a]$, then using the fact that $l$ is odd
$$ 0 = G^{l}(a) \leq G^{l}(x) < G^{l}(\alpha^-) = \alpha^-$$
so $ G^{l}(x) < \alpha^- < x $. 
$\qed$


\subsection{Universal encoding for unimodal maps}  \label{Lambda}

We now introduce a set $\Lambda$ which encodes the topological dynamics of unimodal maps in a universal way.
Several similar approaches are possible (\cite{MT}, \cite{dMvS} among others);  we follow \cite{IP}, \cite{Is}.
 We recall that a smooth map $f:[0,1] \to [0,1]$ is called {\sl unimodal} if it has exactly one critical point $0<c_0<1$, and we are going
to assume $f(0)=f(1)=0$. 
If $x \in [0,1]$ never maps to $c_0$ under $f$, let us define the {\sl itinerary} of $x\in[0,1]$ to be the sequence 
$$i(x)=s_1s_2\dots \ \ \mbox{ with }  \ \ s_i=\left\{
\begin{array}{ll} 
0 & \textup{if }f^{i-1}(x)<c_0\\
1 & \textup{if }f^{i-1}(x)> c_0
\end{array}
\right.
$$
If $x$ eventually maps to $c_0$, let us define $i(x^\pm) := \lim_{y \to x^\pm} i(y)$ where the limit is taken over non-precritical points. 
The {\sl kneading  sequence } of a unimodal map is the binary sequence $K := i(c_1^-)$, where $c_1 = f(c_0)$ is the critical value.

Kneading sequences encode the combinatorics of the orbits and determine the topological entropy.
A nice way to decide whether a given binary sequence $s \in \{0,1\}^\N$ is the kneading sequence of some unimodal map is the following: 
we first associate to $s= s_1s_2\dots  \in \{0,1\}^\N$ the number $\tau (s)\in [0,1]$ defined as
\be\label{eq:enc}
\tau =0.t_1t_2\dots =\sum_{k\geq 1} t_k 2^{-k}, \qquad t_k = \sum_{i=1}^k s_i \, {\rm (mod \  2)}
\ee
One readily checks that the following diagram is commutative:

\be
\label{diagram}
\begin{CD}
[0,1] @>i >> \{0,1\}^\N @>\tau >> [0,1] \\
@Vf VV  @VV \sigma V  @VV T V\\
[0,1] @>i >> \{0,1\}^\N @>\tau >> [0,1] \\
\end{CD}
\ee

\vskip 0.5 cm 
\noindent
where $T$ is the tent map (\ref{tenda}) and $\sigma : \{0,1\}^\N \to \{0,1\}^\N$ is the shift map. Since the critical value is the 
highest value in the image and the map $\tau \circ i$ is increasing (\cite{Is}, Lemma 1.1), one has the characterization 
\begin{proposition}[\cite{IP}, \cite{Is}]
A binary sequence $s\in \{0,1\}^\N$ is the kneading sequence of a continuous unimodal interval map if and only if 
$\tau := \tau (s)$ satisfies 
$$T^k(\tau) \leq \tau \textup{ for all }k\geq 0$$
\end{proposition}

\noindent One is thus led to study the set
\be
\Lambda := \{ x\in [0,1]\, : \, T^k(x)\leq x, \; \forall k\in \N\}
\ee
By considering $T(x) \leq x$ one immediately realizes that  $\Lambda  \setminus \{0\}\subseteq [2/3,1]$. Note moreover that, if the map is $C^1$, the extremal situations in which 
$T^k(\tau)=\tau$
for some $k>0$ correspond to periodic attractors (and thus periodic kneading sequence) 
of
the corresponding smooth unimodal maps (see [dMvS], Chapter IV).

Note that all kneading sequences which arise from unimodal maps can be actually realized by real quadratic polynomials $f_c(z) = z^2 + c$, 
$c \in [-2, \frac{1}{4}]$. By density of hyperbolicity
 (\cite{GS}, \cite{Ly}), aperiodic kneading sequences are realized by exactly one map $f_c$. 
On the other hand, the set of parameters which realize a given admissible periodic kneading sequence is an interval in parameter space, 
and among those maps there is exactly one $f_c$ which has a parabolic orbit.
In conclusion, $\Lambda$ parametrizes the set of topologically unstable parameters, hence it will be called the \emph{binary bifurcation set}.

\section{From continued fractions to kneading sequences} \label{main}
\noindent
We are now ready to prove Theorem \ref{mth}, namely that the map
$\varphi: [0, 1] \rightarrow [\frac{1}{2},1]$ given by 
$ x = [0;a_1, a_2, a_3, ...]
\mapsto \varphi(x) = 0.\underbrace{11\ldots1}_{a_1}\underbrace{00\ldots0}_{a_2}\underbrace{11\ldots1}_{a_3}\ldots $
is an orientation-reversing homeomorphism which takes $\mathcal{E}$ onto $\Lambda\setminus\{0\}$.

\textit{Proof of theorem \ref{mth}.}
The key step is that Minkowski's question mark function conjugates the Farey and tent maps, i.e.
\begin{equation} \label{conj}
?(F(x)) = T(?(x)) \qquad \forall x \in [0,1]
 \end{equation}
Note that $\varphi = ? \circ \Psi_1$, hence $\varphi$ is a homeomorphism of $[0,1]$ onto the image of $\Psi_1$, i.e. $[\frac{1}{2}, 1]$. Moreover, 
by Lemma \ref{characterizeE}-(d),
$$\Psi_1(\mathcal{E}) = \{x \in (0, 1] : F^k(x) \leq x \}$$
and by \eqref{conj}
$$?(\Psi_1(\mathcal{E})) = \{ x \in (0,1] : T^k(x) \leq x \} = \Lambda\setminus\{0\}. \qquad \qed$$

\noindent In the following subsections we will investigate a few consequences of such a correspondence.

\subsection{Construction of $\Lambda$ and bisection algorithm} \label{bisec}

A direct consequence of the theorem is the following algorithm to construct $\Lambda$, which mimics the one used to construct $\cal E$: 
let $\omega=\omega_1\dots \omega_{n-1}1$ be a binary word of length $n$ (ending with $1$) and $d=0.\omega$
the corresponding dyadic rational. Setting $\omega^*=\omega_1\dots \omega_{n-1}0$, we see that $d$ has exactly two representations:
$$
d=0.\,\omega\, {\overline 0} \quad \hbox{and} \quad d= 0.\,\omega^* {\overline 1} 
$$
where, if $\eta$ is any finite binary word, $\overline{\eta}\in \{0,1\}^\N$ denotes the 
 infinite sequence obtained by repeating $\eta$ indefinitely.
 Now define the \textit{dyadic interval} associated to $d$ as the open interval $J_d=(r^-,r^+)$ whose endpoints are the (non-dyadic) rationals
\footnote{Since the correspondence is orientation reversing, the string defining $r^-$ corresponds to the string $A^+$ which defines the upper endpoint of the interval in \cite{CT}, Section 2.2.}

$$r^- =0.{\overline {\omega^*} } \quad \hbox{and} \quad r^+=0.{\overline {\omega{\hat \omega}} },\quad $$

\vskip 0.2cm
\noindent
\noindent
{\sc Example.} 
Take $d=13/16 = \varphi(2/5)$. Since the binary expansions of $d$ are $0.1101{\overline 0}$ and 
$0.1100{\overline 1}$, we get $r^-= 0.{\overline {1100}}=4/5 $ and $r^+=0.{\overline {11010010}}=14/17$. 
(See also the example in Section \ref{pedou}.)

\begin{proposition} \label{Lambda_complement}
 We have
\be
\Lambda = [0,1) \setminus \bigcup_{d\in (0,1]\cap \Q_2} J_d
\ee
with $\Q_2:=\{ k/2^s\, : \, s\in \N, \, k\in \Z\}$. 
\end{proposition}

\begin{proof}
From Theorem \ref{mth} and Equation \eqref{Edef}. 
\end{proof}


\vskip 0.1cm
\noindent
Note that in the previous construction there is substantial overlapping among the $J_d$. It is possible, however, to 
directly produce all connected components of the complement of $\Lambda$ by a bisection algorithm: namely, one can 
produce a new component in the gap between two previously computed components.

\begin{definition} 
Let $I = [a, b] \subseteq [0,1]$ be a closed interval, $a \neq b$, and let the binary expansions\footnote{If $a$ or $b$ are dyadic numbers, take the expansion of $a$ which terminates with infinitely many zeros and the expansion of 
$b$ with infinitely many ones} of 
$a$ and $b$ be 
$$a = 0.a_1a_2\ldots \qquad b = 0.b_1b_2\ldots$$
Let $n := \min\{ i : a_i \neq b_i \}$. The \emph{binary pseudocenter} of $I$ is the dyadic rational
\be\label{psece}
d(I) := 0.a_1\ldots a_{n-1} 1
\ee
\end{definition}

By translating  one gets the following description of $\Lambda$:

\begin{proposition}
Let $F_0 = [0, 1]$, and for any $n$ define recursively
$$\mathcal{G}_n := \{ I \mid I \textup{ is a connected component of }F_n,\ I\textup{ not an isolated point } \}$$
$$F_{n+1} := F_n \setminus \bigcup_{I \in \mathcal{G}_n} J_{d(I)}$$
Then $$\Lambda = \bigcap_{n \in \mathbb{N}} F_n$$
\end{proposition}
 
\begin{proof}
From Proposition 2.11 of \cite{CT} and Theorem \ref{mth}. 
\end{proof}


\subsection{Period doubling}\label{pedou}

Period doubling bifurcation is a well known phenomenon which occurs with a universal structure in families of unimodal maps. 
Period doubling cascades are in one-to-one correspondence with \emph{periodic windows} consisting of isolated points in the set $\Lambda \setminus \{0\}$, as shown in 
\cite{AC1}, \cite{IP}. We will see how similar windows of isolated points occur in $\EE$, and in both cases
the combinatorial pattern can be understood in terms of the Thue-Morse sequence ${\bf t}:=(t_n)_{n=0}^{+\infty}=(0, 1, 1, 0, 1, 0, 0, 1, 1, ...)$.
Recall that ${\bf t}$ can be defined in terms of the power series expansion 
 \be\label{eq:TM}
\prod_{k\geq 0} (1-z^{2^k})= 1-z-z^2+z^3+\cdots = \sum_{n\geq 0}
(-1)^{t_n}z^n \ee

\begin{definition}
Define the map $\Delta$ on the set of finite binary words as 
\be\label{uno}
\Delta (\eta) := \eta \,1\, {\hat \eta}
\ee
Moreover, for any finite binary word $\eta$, let $\tau_0(\eta)$ be the rational number $0. \overline{\eta 0}$, and let for any $j$
\be\label{due}
\quad \tau_j(\eta):=\tau_0(\Delta^j(\eta)) \qquad 
 \tau_\infty(\eta) = \lim_{j\to \infty} \tau_j(\eta)\quad
 \ee
The \emph{periodic window} associated to $\eta$ is 
$$W_\eta = [\tau_0(\eta), \tau_\infty(\eta))$$
\end{definition}

Let $x \in \Lambda$ be a point with periodic binary expansion, and let $x = 0.\overline{\eta0}$ with $\eta0$ minimal period. Then 
the periodic window associated to $\eta$ contains exactly the sequence of isolated points just constructed:
$$\Lambda \cap W_\eta = \bigcup_{j \geq 0} \tau_j(\eta)$$

\noindent
Let us call {\sl generating number} of the window $W_\eta$ the pseudocenter $d_0(\eta)$ of the first interval $[\tau_0(\eta), \tau_1(\eta)]$.

\vskip 0.4cm
\noindent
\noindent
{\sc Examples.} 

\noindent 
If $\eta = \epsilon$, the empty word, we obtain  the sequence $\{\tau_j(\epsilon)\}_{j\geq 0}=\{0, \frac 2 3, \frac 4 5, \frac {14}{17}, \dots\}$ which converges to $\tau_\infty(\epsilon) = 0.824908...$, and the generating number of $W_\epsilon$ is $d_0(\epsilon)=1/2$.
Note that the binary expansion of $\tau_\infty(\epsilon) = 0.11010011001011...$ is the Thue-Morse sequence.

\noindent
Taking $\eta = (11)$ we get the sequence $\{\tau_j(11)\}_{j\geq 0}=\{\frac 6 7, \frac 8 9, \frac {58}{65}, \dots\}$ which converges to the irrational number $\tau_\infty(11)=0.892360...$ and 
belongs 
to the next-to-largest window $W_{11}$ in $\Lambda$. Its generating number is $d_0(11)=7/8$.






\vskip 0.3 cm

Windows of isolated points are present also in the set $\cal E$ and
can be generated with a similar procedure, introduced in \cite{CMPT}
and \cite{CT} without connection to bifurcations of unimodal
maps. Indeed, let $I_r$ be a maximal quadratic interval, and let $r =
[0; S_0] = [0; S_1]$ be the continued fraction expansions of its
pseudocenter, in such a way that $S_0$ has even length and $S_1$ has
odd length. Let us define $\Sigma_0 := S_0$, $\Sigma_1 := S_1$ and
$\Sigma_{n+1}$ is generated from $\Sigma_n$ via the substitutions $S_1
\rightarrow S_1 S_0$, $S_0 \rightarrow S_1 S_1$. If we denote

$$\alpha_n(r) := [0; \overline{\Sigma_n}] \qquad \alpha_\infty(r) := \lim_{n \to \infty} \alpha_n(r)$$
it follows immediately from \cite{CT}, Proposition 3.9 that

\begin{lemma} 
For each $I_r$ maximal quadratic interval of pseudocenter $r$, 
$$\EE\cap (\alpha_\infty(r), \alpha_0(r)] = \bigcup_{n \geq 0} \alpha_n(r)$$
\end{lemma}

Limit points $\alpha_\infty(r)$ of these cascades have an aperiodic continued fraction expansion 
with a common combinatorial structure which is again related to the Thue-Morse sequence ${\bf t} = (t_n)_0^\infty$: indeed, for any $r$
$$\alpha_\infty(r) = [0; S_{k_1}, S_{k_2}, S_{k_3}, \dots]$$
where the sequence $K :=(k_1, k_2, k_3,...)= (1,0,1,1...)$ is such that $\tau(K) = \sum_{j=0}^{+\infty} t_j 2^{-j}$.
For instance, if $r = \frac{1}{2}$, 
then\footnote{Notice that the sequence of partial quotients of $\alpha_\infty(\frac{1}{2})$ coincides with the first difference sequence
of \newline
$c =1\;3\;4\,5\;7\;9\;11\;12\;13\; \cdots$
defined as $n\in {\bf \rm c} \iff 2n \notin {\bf \rm c}$.} 
$\alpha_\infty(\frac{1}{2}) = [\,0\,; 2, 1 ,1 ,2 , 2, 2, 1, 1 ,2, \dots ]
 \cong 0.38674997...$ is the accumulation point considered in \cite{CMPT} and \cite{KSS}.

We end this section with a precise characterization of the points $\{ \tau_j(\eta) \}_{j \ge 0}$ in $\Lambda$ and of their 
accumulation point $\tau_\infty(\eta)$, which we prove to be a transcendental number.

\begin{lemma}\label{lemmino} Let $\eta= \eta_1\dots\eta_p$ be a finite binary word of length $p\ge 0$ such that 
$\tau_0(\eta) \in \Lambda$ 
and define recursively the integers:
\begin{itemize}
\item $P_0 := \sum_{i = 1}^p \eta_i 2^{p+1-i}$
\item $P_{j+1} := (P_j + 2)(2^{2^j(p+1)}-1)$
\item $Q_j := 2^{2^j(p+1)} -1$
\end{itemize}
Then 
$\tau_j(\eta) \in \Lambda$ for every $j$, and $\tau_{j}(\eta)=P_j/Q_j$.
Moreover, the binary pseudocenters are given by
$$d_j(\eta) = 2^{-2^j(p+1)}(P_j + 1) = \frac{P_j + 1}{Q_j + 1}$$
\end{lemma}



\begin{proof} 
We need the following elementary identity
\be\label{bina}
0.\overline{\omega_1\dots \omega_s} = \frac {2^s}{2^s-1} \sum_{i=1}^s \frac {\omega_i}{2^i}
\ee
for any binary word $\omega_1\dots \omega_s$.
$P_0$ is the integer whose binary expansion is $\eta 0$, hence from \eqref{bina} $\tau_0 = P_0/Q_0$.
Moreover, if $P$ is the integer whose binary expansion is $\xi0$, with $\xi$ a binary string, then the 
number whose binary expansion is $\xi 1 \hat{\xi} 0$ is $P' = (P+2)(2^{|\xi|+1}-1)$.
By the recursive formulas \eqref{uno} and \eqref{due} and identity \eqref{bina} one has
$\tau_j = P_j /Q_j$ for all $j \geq 1$.
Finally, the binary pseudocenter of the interval $[0.\overline{\xi0}, 0.\overline{\xi1\hat{\xi}0}]$ is 
$d = 0.\xi1 = (P+1)2^{-(|\xi|+1)}$
hence the lemma follows by taking $P = P_j$.
\end{proof}
Furthermore, let us consider the generating function (cfr. Equation (\ref{eq:TM}))

\be \Xi (z): = \prod_{k\geq 0} (1-z^{2^k})
\ee
This function satisfies the functional equation
$\Xi (z)=(1-z) \Xi (z^2)$, from which we see that all its zeroes lie on the unit circle and are actually dense there (see also \cite{Is}). 

\begin{proposition} \label{xi}
Let $\eta =\eta_1\dots\eta_p$ be a finite binary word of length $p\ge 0$ such that 
$\tau_0(\eta) \in \Lambda$. 
Then
\begin{equation} \label{transc}
\tau_\infty(\eta)=1- \Big(1 - d_0(\eta)\Big)\ \Xi \left( \frac 1 {2^{p+1}} \right)
\end{equation}
In particular, $\tau_\infty(\eta)$ is transcendental. 
\end{proposition}
\begin{proof} Setting $\tau_{j}(\eta)=P_j/Q_j$
and using Lemma \ref{lemmino} we get for $j\geq 1$
$$
Q_j-P_j=(2^{(p+1)2^{j-1}}-1)(Q_{j-1}-P_{j-1})=\cdots = (Q_0 - P_0)\, \prod_{k=0}^{j-1}(2^{(p+1)2^{k}}-1)\, .
$$
Therefore
$$
\tau_{j}(\eta)=1- (Q_0 - P_0)\, \frac {\prod_{k=0}^{j-1}(2^{(p+1)2^{k}}-1)} {2^{(p+1)2^{j}}-1} = 1 - \frac{\prod_{k = 0}^{j-1}(1 - 2^{-2^l(p+1)})}{1 - 2^{-2^j(p+1)}}(1-d_0(\eta))
$$
where we used the relation $Q_0 - P_0 = 2^{p+1}(1 - d_0(\eta))$ given by Lemma \ref{lemmino}. The claimed identity follows by taking the limit. Finally, Mahler proved in (\cite{M}, page 363) that $\Xi (a)$ is transcendental for every algebraic number $a$, with $0<|a|<1$. 
\end{proof}

\noindent
Examples of $\tau_\infty(\eta)$ are 
$$
\tau_\infty(\epsilon)=1- \frac 1 {2}\, \Xi \left( \frac 1 {2} \right),\quad \tau_\infty(11) = 
1- \frac 1 {8}\, \Xi \left( \frac 1 {8} \right) \quad \text{and} \quad \tau_\infty(1101) = 1 - \frac{5}{32}\Xi\left(\frac{1}{32}\right) 
$$

\vskip 0.5 cm
\noindent
A corresponding transcendence result can be given for $\mathcal{E}$. Namely, accumulation points of period 
doubling cascades $\alpha_\infty(r)$ have continued fraction expansion given by the substitution rule 
explained above, hence they are transcendental by \cite{AB}. 

It has been widely conjectured that all numbers with non-eventually periodic bounded partial quotients 
are transcendental. Let us point out that such a statement is equivalent to the transcendence of all 
non-quadratic irrationals in $\mathcal{E}$. Similarly, one might ask whether all irrational points in $\Lambda$ are transcendental
 as well.

\subsection{The topology of bifurcation sets}  \label{top}

By using the bisection algorithm and Proposition \ref{xi} it is immediate to determine the topology of bifurcation sets, namely

\begin{proposition} \label{topology}
Let $\mathbb{Q}_1$ be the set of rational numbers with odd denominator. Then 
\begin{enumerate}
\item The set $\Lambda \cap \mathbb{Q}_1$ is dense in $\Lambda$.
\item The derived set $\Lambda'$ (i.e. $\Lambda$ minus its isolated points) is a Cantor set. 
As a corollary, 
$\Lambda'$ and $\Lambda$ have cardinality $2^\omega$.
\item Let $C$ be the usual Cantor middle-third set, let $\Omega_k = (\alpha_k, \beta_k)$, $k \geq 1$ be the connected components of the complement 
$[0, \infty) \setminus C$. For each $k$, pick a countable set of points $P_k \subseteq \Omega_k$ which accumulate only on $\alpha_k$. Then 
$\Lambda$ is homeomorphic to $C \cup \bigcup_{k \geq 1} P_k$.
\item Any open neighbourhood of an element $\tau \in \Lambda \setminus \mathbb{Q}_1$ contains a subset of $\Lambda$ which is homeomorphic to $\Lambda$ itself (fractal structure). 
\end{enumerate}
\end{proposition}

\noindent 
Note that, even though the proposition was stated for $\Lambda$, the same result holds for $\mathcal{E}$ since they are homeomorphic. 
The only change is that you have to replace $\mathbb{Q}_1$ by the set of 
quadratic irrationals with purely periodic continued fraction expansion.

\begin{proof}
\begin{enumerate}
\item $\Lambda$ is closed with no interior by Proposition \ref{Lambda_complement}, hence the endpoints of the connected components of its complement are dense in it.
The connected components of the complement of $\Lambda$ are intervals of type $J_d$, so their endpoints have purely periodic binary expansion, hence
they belong to $\mathbb{Q}_1$.
\item
 By virtue of the homeomorphism $\varphi$ of Theorem \ref{mth}, 
it is equivalent to prove the same result for $\mathcal{E}$. 
Let $\PP$ denote the points of $\EE$ which are periodic under $G$
$$\PP:=\{\lambda \in \EE \ : \ G^{k_0}(\lambda)=\lambda \ \mbox{ for some } k_0\in \N\}$$
and let us define the set of  {\it primitive
  elements} as
$$\PP_0:=\{\lambda \in \PP \ :  \lambda \mbox{ has even minimal period} \}.$$ 

\begin{lemma} \label{isolated}
The set of isolated points of $\mathcal{E}$ is precisely $\mathcal{P}\setminus \mathcal{P}_0$.
\end{lemma}

\begin{proof} 
If $\alpha$ is not a limit point of $\mathcal{E}$ then it
is the separating element between two adjacent maximal quadratic
intervals $J_0$ and $J_1$, so $\alpha$ is the left endpoint of the
rightmost interval and has odd minimal period (see \cite{CT}).
\end{proof}
Moreover, if $\lambda\in \mathcal{P}$, we denote with
$\alpha_\infty(\lambda)$ the limit point of the period doubling
cascade generated\footnote{So for instance $\alpha_\infty(g)=
  0.38674997...$ (\cite{CMPT}, page 24).} by $\lambda$ and let
  $W_\lambda:=(\alpha_\infty(\lambda),\lambda)$ be the corresponding period doubling
  window.  We also set
  $\cal{P}_\infty:=\{\alpha_\infty(\lambda) \ : \ \lambda \in
  \mathcal{P}\}$, hence by Lemma \ref{isolated}

$$\mathcal{E}' = [0, 1) \setminus \bigcup_{\lambda \in \mathcal{P}_0} (\alpha_\infty(\lambda),\lambda).$$
Now, no two intervals of the form $(\alpha_\infty(\lambda), \lambda)$ with $\lambda \in \mathcal{P}_0$ are adjacent to each other, 
since the right endpoint $\lambda$ is always a quadratic irrational while $\alpha_\infty(\lambda)$ is always transcendental 
(see Proposition \ref{xi}), therefore $\mathcal{E}'$ has no isolated points and it is a Cantor set.

\item Every closed subset of the interval with no interior and no isolated points is homeomorphic to the usual Cantor middle-third set 
via a homeomorphism of the ambient interval. Such an extension can be chosen so that it maps any period doubling cascade to some $P_k$. 
\item Immediate from 3.
\end{enumerate}

\end{proof}

Let us remark that similar results have been obtained for the set of univoque numbers in \cite{EHJ}, \cite{KL2}, 
and by Allouche and Cosnard for a similar number set (\cite{A}, \cite{AC1}, \cite{AC2}).
For details on the relations between these sets see Section \ref{univ}. 

Moreover, $\EE$ shares some features with the Markov spectrum, since the Gauss map is the first-return map on a section 
of geodesic flow on the modular surface. Even though the two sets are not homeomorphic, property 2 also holds for the 
Markov spectrum \cite{Mo}.

\subsection{Hausdorff dimension} \label{HD}

In \cite{CT} it is proved that $dim_H \EE = 1$ by providing estimates on the Hausdorff dimension of its segments. 
More precisely, 
for every $K\in \N_+$ we consider
$$\mathcal{E}_K:=\mathcal{E}\cap [1/(K+1), 1/K], \ \ B_K:=\{x=[0;a_1,
  a_2, a_3,...] : a_i\leq K \ \forall i\in \N\}, \ \ \Psi_K(x) :=
\frac{1}{K+x}
$$
and it is easy to check that
\begin{equation}\label{eq:hdk}
\Psi_K(B_{K-1})\subset \mathcal{E}_K\subset B_K,
\end{equation}
therefore\footnote{In fact, the asymptotics of \cite{He}
  yields the estimate $dim_H \mathcal{E}_K= 1-6/\pi^2K + o(1/K)$.}
$dim_H \mathcal{E}=\lim_{K \to \infty} dim_H \EE_K = \lim_{K \to \infty} dim_H B_K = 1$.

We can use our dictionary to obtain analogous results in the linear
setting, where in fact one can explicitely compute the Hausdorff
dimension of segments of $\Lambda$. Let $\Lambda_K:=\Lambda
\cap[1-2^{-K}, 1-2^{-K-1}]$, $\Phi_K(x) := x \mapsto 1-2^{-K}x$ and
$$C_K:=\{ x\in [1/2,1]: \ \underline{x} \mbox{ does not contain
   sequences of } K+1 \mbox{ equal digits} \};$$
using the correspondence of Theorem \ref{mth}, the inclusions \eqref{eq:hdk} become
$\Phi_K(C_{K-1})\subset \Lambda_K\subset C_K$,  and thus 
\begin{equation}\label{eq:hdkl}
dim_H C_{K-1} \leq dim_H \Lambda_K \leq dim_H C_K.
\end{equation}
The set $C_K$ is self-similar, therefore its Hausdorff dimension can be
computed by standard techniques (see \cite{F}, Theorem 9.3). More
precisely, if $a_K(n)$ is the number of binary sequences of $n$ digits
whose first digit is $1$ and do not contain $K+1$ consecutive equal
digits, one has the following linear recurrence:
\footnote{Sequences satisfying this relation are known as \emph{multinacci} sequences, 
being a generalization of the usual Fibonacci sequence; 
the positive roots of their characteristic polynomials are Pisot numbers.} 
\begin{equation}\label{eq:multinacci}
a_K(n+K)=a_K(n+K-1)+...+a_K(n+1)+a_K(n)
\end{equation}
which implies  
\begin{proposition}\label{pisot}
For any fixed integer $K\geq 2$ the Hausdorff dimension of $C_K$ is 
$\log_2(\lambda_k)$, where $\lambda_K$ is the only
positive real root of the characteristic polynomial
$$P_K(t):= t^K- (t^{K-1}+...+t+1)$$
\end{proposition} 

As a consequence, a simple estimate on the unique positive root of $P_K$ yields 
$$ \dim_H \Lambda = \lim_{K\to +\infty} dim_H C_K =1$$


\section{Kneading sequences, complex dynamics and univoque numbers} \label{last}

The goal of this section is to explore the relations between the bifurcation sets we described before
and other well-known sets for which a combinatorial description can be given, namely the real slice of the 
Mandelbrot set and the set of univoque numbers.

\subsection{Relation to the Mandelbrot set} \label{mandel}

Let us recall that the Mandelbrot set encodes the dynamical properties of the family of quadratic polynomials
$$p_c(z) := z^2 + c \qquad c \in \mathbb{C}$$
namely it can be defined as 
$$\mathcal{M} := \{ c \in \mathbb{C} \mid  p_c^n(0) \textup{ does not tend to } \infty \}$$
A combinatorial model of the Mandelbrot set can be 
constructed by using the theory of invariant quadratic laminations, as developed by Thurston \cite{Th}.

Consider the unit disk in the complex plane with the action of the doubling map 
$$f(z) = z^2$$
on the boundary. A \emph{leaf} $\mathcal{L}$ is a simple curve embedded in the interior of the circle joining two points on the boundary.
It is usually represented as a geodesic for the hyperbolic metric in the Poincar\'e disk model. 
One can extend the action of $f$ to the whole lamination: the image of a leaf $\mathcal{L}$ is, by definition,
the leaf which connects the images of the endpoints of $\mathcal{L}$. 
We define the \emph{length} $\ell(\mathcal{L})$ of a leaf to be the (euclidean) length of the shortest arc of circle delimited 
by its endpoints, measured on the boundary circle and normalized in such a way that the whole circumference has length $1$ 
(hence, for any leaf $\mathcal{L}$, $0 \leq \ell(\mathcal{L}) \leq \frac{1}{2}$).
A leaf is \emph{real} if it is invariant with respect to complex conjugation. The doubling map $f$ preserves the set of real leaves.

A \emph{lamination} is a closed subset of the disk which is a union of leaves, and 
such that any two leaves in the lamination can intersect only on the boundary of the disk.
A \emph{quadratic invariant lamination} is a lamination whose set of leaves is 
completely invariant with respect to the action of the doubling map (we refer to \cite{Th}, page 66 for a complete definition
of this invariance). 
Given an invariant quadratic lamination, its longest leaves are called its \emph{major leaves}, 
and their image the \emph{minor leaf}. There can be at most $2$ major leaves, and they both have 
the same image. In the well-known theory of Douady, Hubbard and Thurston, the leaves of an invariant quadratic lamination 
define an equivalence relation on the boundary circle and the quotient space is a model for the Julia set of a certain quadratic polynomial. 

The \emph{quadratic minor lamination} $QML$ is the set of all minor leaves of any quadratic invariant lamination
$$QML := \{ \mathcal{L} \mid \exists \textup{ an invariant quadratic lamination }L_0 \textup{ s.t. }\mathcal{L}\textup{ is its minor leaf} \}$$

By (\cite{Th}, Theorem II.6.8), different minor leaves do not cross, and $QML$ is indeed a lamination.
Similarly to the Julia set case, the space obtained by identifying points on the disk which belong to the same leaf of $QML$
is a model of the classical Mandelbrot set.  

We will consider the set $RQML \subseteq QML$ of \emph{real leaves} inside the quadratic minor lamination:
under the previous correspondence, they correspond to real points on the boundary of the Mandelbrot set.
Moreover, 
let $\mathcal{R} = RQML \cap S^1$ be the set of endpoints of all leaves
in $RQML$. Since $RMQL$ is symmetric, it is enough to consider the ``upper half'' $\mathcal{R} \cap [0, \frac{1}{2}]$.

\begin{figure}[h]
  \centering
  \subfloat[The invariant quadratic lamination for the parameter $\frac{5}{6} \in \Lambda$. In red the minor leaf joining $(\frac{5}{12}, \frac{7}{12})$, 
in blue the major leaves. The corresponding parameter in $\mathcal{E}$ through the dictionary of Section \ref{main} is $\frac{3-\sqrt{5}}{2}$, the square of the golden mean.]{\label{fig:quadlam}\includegraphics[width=0.45\textwidth]{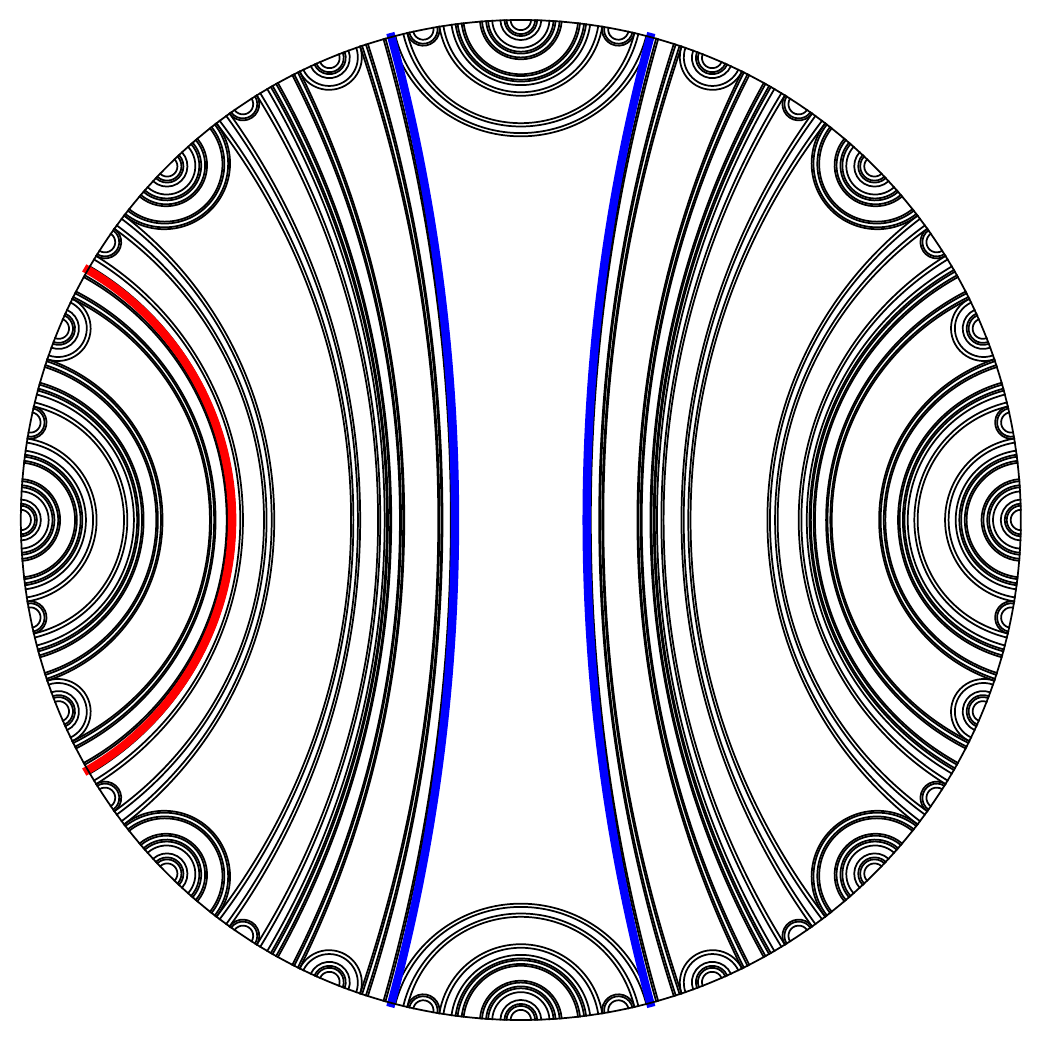}}
  \qquad
  \subfloat[The quadratic minor lamination $QML$ is the collection of all minor leaves. Real leaves are drawn in red, 
and the set $\mathcal{R}$ is the set of all endpoints of red leaves. 
A thicker red line represents the minor leaf corresponding to the lamination to the left. ]{\label{fig:QML}\includegraphics[width=0.45\textwidth]{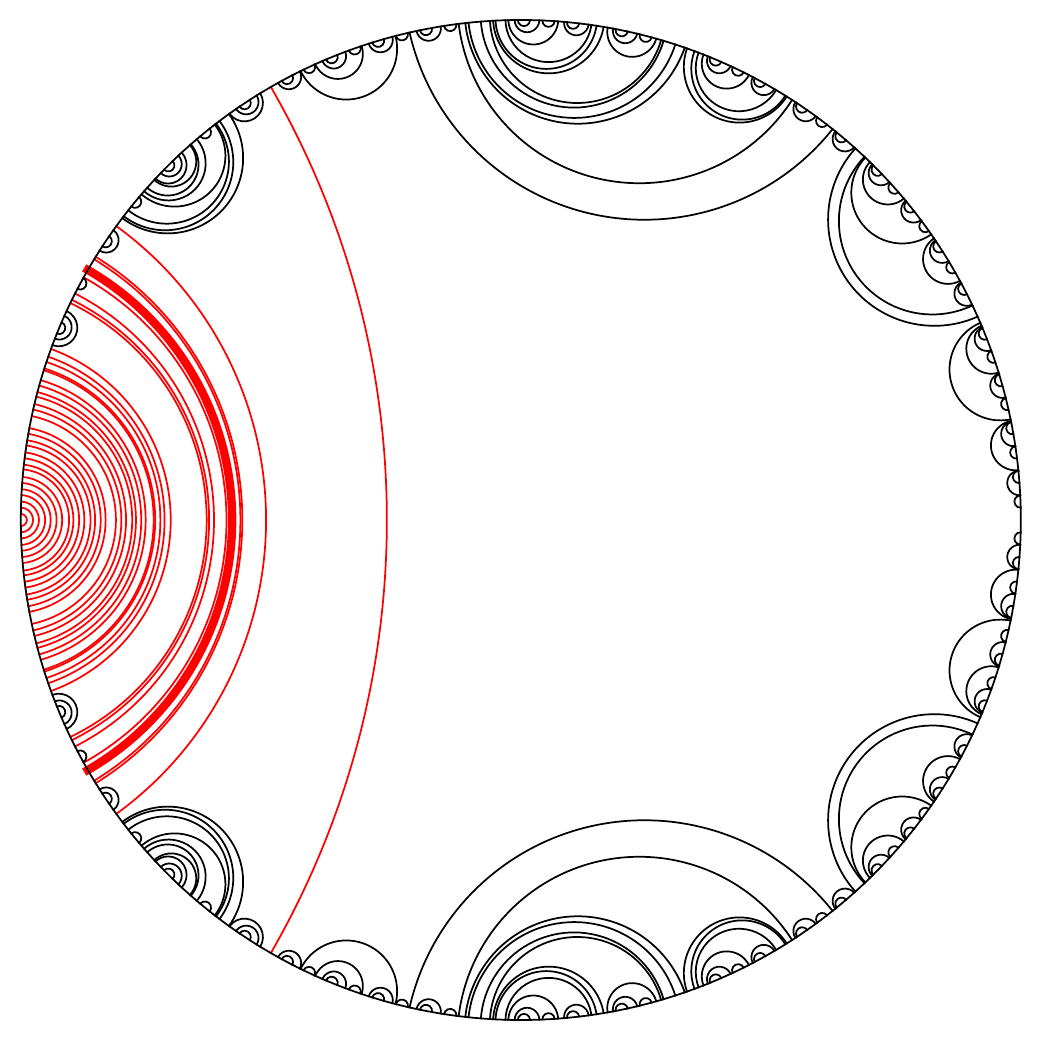}}
 \end{figure}



 
\begin{proposition}\label{simple}
The map $x \mapsto 2x$ maps $\mathcal{R} \cap [0, \frac{1}{2}]$ bijectively onto $\Lambda$. Hence
$$\mathcal{R} = \left( \frac{1}{2}\Lambda \right) \cup \left( 1 - \frac{1}{2}\Lambda \right)$$
\end{proposition}

Proposition \ref{simple}, together with the fact that $dim_H\Lambda=1$, yields a new proof of the main result in \cite{Za}

\begin{corollary}
The Hausdorff dimension of $\mathcal{R}$ is $1$.
\end{corollary}

Because of the fact that the set of rays which possibly do not land has zero capacity and by Makarov's dimension theorem 
\cite{Ma}, we proved 

\begin{corollary}  \label{mandeldim}
The intersection of the boundary of the Mandelbrot set with the real line has Hausdorff dimension $1$ . 
\end{corollary}


\textit{Proof of Proposition \ref{simple}.}
The first ingredient is the fact that the action induced by the doubling map on the lengths of leaves is 
essentially given by the tent map, namely for any leaf $\mathcal{L}$
\begin{equation}
\label{eq:tent}
2 \ell(f(\mathcal{L})) = T(2 \ell(\mathcal{L}))
\end{equation}

Let $x \in \mathcal{R} \cap [0, \frac{1}{2}]$ be the endpoint of a minor leaf, let $m$ be the 
minor leaf and $M$ be one of the corresponding major leaves. By maximality of $M$ and invariance of the 
quadratic lamination, $\ell(M) \geq \ell(f^k(M))$ for all $k \in \mathbb{N}$, hence by \eqref{eq:tent}
\begin{equation} \label{eq:max}
2 \ell(M) \geq T^k (2 \ell(M)) \quad \forall k \in \mathbb{N}
\end{equation}
From this and the fact that $x = 2/3$ is a fixed point for $T$ 
it follows that 
the major leaf is longer than $\frac{1}{3}$, hence
$$\ell(m) = \frac{1}{2}T(2\ell(M)) = 1 - 2\ell(M)$$
Moreover, by symmetry of the minor leaf with respect to complex conjugation, 
$$\ell(m) = 1 - 2x$$
hence $x = \ell(M)$ and by \eqref{eq:max}
$$2x \geq T^k(2x) \quad \forall k \in \mathbb{N}$$
hence $2x$ belongs to $\Lambda$. Conversely, if $y = 2x \in\Lambda$, one can construct the real leaf $m$ 
which joins $x$ with $1-x$. The fact that $m$ is a minor leaf follows from the criterion (\cite{Th}, page 91):
a leaf $m$ is the minor leaf of an invariant quadratic lamination if and only if 
\begin{enumerate}
\item[(i)] all forward images have disjoint interior;
\item[(ii)] no forward image is shorter than $m$;
\item[(iii)] if $m$ is non-degenerate, then $m$ and all leaves on the forward orbit 
can intersect the (at most 2) preimage leaves of $m$ of length at least $\frac{1}{3}$
only on the boundary.
\end{enumerate}

Indeed, (i) follows since images of real leaves are real; (ii) follows from \eqref{eq:max}, 
and (iii) follows from the fact that the preimages of a real leaf of length less than $\frac{1}{3}$ are also real. 
$\qed$

\subsection{An alternative characterization and univoque numbers} \label{univ}

\begin{definition}
A real number $q \in (1, 2)$ is called \emph{univoque} if $1$ admits a
{\sl unique} expansion in base $q$, i.e.  if there exists a
\emph{unique} sequence $\{c_k\} \in \{0, 1\}^{\mathbb{N}}$ such that
\begin{equation}\label{eq:duality} 
\sum_{k\geq 1} c_k\, q^{-k} = 1
\end{equation}
The set $\mathcal
   U$ of all univoque numbers is called the {\sl univoque set}.
\end{definition}

In the 90's, 
\cite{EHJ} discovered that
there are infinitely many univoque numbers. In fact \cite{EJK} showed
that, given any binary sequence $(c_k)_{k=1}^\infty$, then equation
\eqref{eq:duality} defines a univoque number if and only if the
following condition is met 
\begin{equation}\label{eq:admissible}
\left\{
\begin{array}{ll}
(c_{k+1},c_{k+2},c_{k+3},...)< (c_1,c_2,c_3,...) & \mbox{ if }c_k=0\\
 (c_{k+1},c_{k+2},c_{k+3},...)>(\hat c_1,\hat c_2,\hat c_3,...)& \mbox{ if }c_k=1
\end{array}
\right. \ \ \ \ \ \  \forall k\geq 1
\end{equation}
where the minor and major signs refer to the lexicographical order and $\hat c=1-c$ as usual.

Binary sequences which satisfy equation \eqref{eq:admissible} are
called {\em admissible}.  It is not difficult to establish (by means of the so called {\sl greedy} algorithm)  that equation
  \eqref{eq:duality} defines an order preserving bijection between
  $\mathcal{U}$ and the set of all admissible sequences.

   The properties of $\mathcal U$ were studied by
   various authors, leading in particular to the determination (in
   \cite{KL1}) of the smallest element, $q_1=1.78723...$, which is not
   isolated in $\mathcal U$ and whose corresponding admissible sequence is the shifted Thue-Morse sequence $(t_n)_{n\geq 1}$. The geometric
   structure of $\mathcal U$ is specified by the following properties:
\begin{itemize}
\item the set $\mathcal U$ has $2^\omega$ elements;
\item the set $\mathcal U$ has zero Lebesgue measure;
\item the set $\mathcal U$ is of first category;
\item the set $\mathcal U$ has Hausdorff dimension $1$.
\end{itemize}
For an account of 
these and other properties of the set $\mathcal U$ see the recent paper \cite{KL2} and references therein.

In a series of papers (see  \cite{A}, \cite{AC1}, \cite{AC2}), J.-P. Allouche and M. Cosnard introduced the number set
\be
\Gamma :=  \{ x\in [0,1]\, : \, 1-x \leq \{2^k x\} \leq x, \; \forall k\in \N\}
\ee
and recognized (\cite{AC2}, Proposition 1) admissible binary sequences to be in a one-to-one correspondence with the elements of $\Gamma$
 which have non-periodic expansions. It is not difficult to realize
 that  the set $\Gamma$ essentially coincides with $\Lambda$, namely  

\begin{lemma}\label{GL}
\be
\Lambda\setminus\{0\} = \Gamma
\ee
\end{lemma}
\noindent
\begin{proof}
Let $x = 0.\omega_1 \omega_2 \dots$ be a binary expansion of $x \in [0, 1]$. From (\ref{tenda}) it follows that
$$
T^k(x) = \left\{ \begin{array}{ll} 0.\omega_{k+1} \omega_{k+2} \dots & \textup{if }\sum_{i=1}^k \omega_i=0 \, {\rm (mod\ 2)}\\
				     0.\hat{\omega}_{k+1} \hat{\omega}_{k+2} \dots & \textup{if }\sum_{i=1}^k \omega_i=1 \, {\rm (mod\ 2)}\end{array} \right.
				     $$
Define
$$
S_k(x) := \max\{ \{2^kx\}, 1-\{ 2^kx\} \} = \left\{ \begin{array}{ll} 0.\omega_{k+1} \omega_{k+2} \dots & \textup{if }\omega_{k+1} = 1 \\
                     0.\hat{\omega}_{k+1} \hat{\omega}_{k+2} \dots & \textup{if }\omega_{k+1} = 0 \end{array}\right.
$$
then $\Gamma= \{ x \in [0, 1] : S_k(x) \leq x,  \; \forall k \in \N \}$ and since $T^k(x) \leq S_k(x)$ we have 
$\Gamma \subseteq \Lambda$.

\noindent
Let now $x \in \Lambda \setminus \{0\}$, and let $I = \{i_k\}_{k \in \N} = \{ i \in \N : T^i(x) = 0.1\dots \}$ ($0 \in I$
since $x \neq 0$). Given $l \in \N$, either $l \in I$ or $l \notin I$. If $l \in I$, then $S_l(x) = T^l(x) \leq x$.
If instead $l \notin I$, then there exists a unique $k$ s.t. $i_k < l < i_{k+1}$. Thus
$$T^{i_k}(x) = 0.\underbrace{1  \dots  1}_{i_{k+1} - i_k} 0\dots\, , \qquad T^l(x) = 0.\underbrace{0\dots0}_{i_{k+1}-l}1\dots $$ 
and hence 
$$S_l(x) = 0.\underbrace{1\dots1}_{i_{k+1}-l}0 \ldots < 0.\underbrace{1\dots1}_{i_{k+1}-i_k}0 \ldots = T^{i_k}(x) \leq x$$
\end{proof}

Therefore $\Lambda$ can be given the following arithmetic interpretation (see also \cite{AC2}):

\begin{proposition}\label{1to1} There is a one-to-one correspondence between $\mathcal U$ and $\Lambda\setminus \mathbb{Q}_1$. More specifically, a
number $\tau =\sum_{k\geq 1} c_k \,2^{-k}$ with non-periodic binary expansion belongs to $\Lambda$ if and only if  $1=\sum_{k\geq 1} c_k \, q^{-k}$ for some $1<q<2$ and the expansion is unique.
\end{proposition} 

\begin{proof}
It is an easy consequence of condition \eqref{eq:admissible} that a
sequence $(c_k)$ is admissible iff 
$$
(\hat c_1,\hat c_2,\hat c_3,...)<(c_{k+1},c_{k+2},c_{k+3},...)< (c_1,c_2,c_3,...) \qquad \forall k \geq 1
$$
In other words, $(c_k)$ is admissible iff the real value $\tau =\sum_{k\geq 1} c_k \,2^{-k}$
is a nonperiodic element of $\Gamma$; thus our claim 
follows from Lemma \ref{GL}.
\end{proof}

\noindent
It has been shown in \cite{AC3} that $q_1 \in \mathcal U$ is transcendental. The argument used there can be adapted to show
that the image in $\mathcal U$ of all the numbers $\tau_\infty (\eta)$ dealt with in Proposition \ref{xi} are transcendental as well. 

\begin{proposition}\label{trascendent}
Let $\tau_\infty(\eta) \in \Lambda$ be an accumulation point of a period doubling cascade, and let $q \in [1, 2]$ 
be the corresponding univoque number. Then $q$ is transcendental.
\end{proposition}

\begin{proof}
Let $\eta$ be a finite binary word of length $p$, and write $\tau_\infty(\eta) = \sum_{k \geq 1} c_k 2^{-k}$ and $1-d_0(\eta) = \sum_{i=1}^p a_i 2^{-i}$, with $c_k, a_i \in \{0, 1\}$. Define the formal power series $A(z):= \sum_{k \geq 1} (1-c_k)z^k$ and $B(z) := (a_1 z + \dots + a_p z^p)\Xi(z^{p+1})$. By Equation 
\eqref{transc} $A(\frac{1}{2}) = B(\frac{1}{2})$, hence $A(z) = B(z)$ as formal power series, since $1 - \tau_\infty(\eta)$ has a unique binary 
expansion being transcendental. As a consequence, $A(\frac{1}{q}) = B(\frac{1}{q})$, and since $\sum_{k \geq 1} c_k q^{-k} =1$, 
$A(\frac{1}{q}) = \frac{1}{q-1} -1$. If $q$ is algebraic, then $A(\frac{1}{q})$ is algebraic, hence also $B(\frac{1}{q})$ and 
$\Xi(\frac{1}{q^{p+1}})$ are algebraic, contradicting Mahler's criterion (\cite{M}, page 363).
\end{proof}

\noindent
From our construction it follows that rational elements of $\Lambda$
correspond to algebraic elements in $ \mathcal U$ (and are in fact
dense in $ \mathcal U$ \cite{dV}); rephrasing the question posed at
the end of Section \ref{pedou}, one might ask whether these are the
only non-transcendental elements of $\mathcal U$.


\begin{remark} 
Other geometric properties of the set $\mathcal{U}$ also follow easily from our correspondence.
Let $u:\Lambda \setminus \mathbb{Q}_1 \to [1,2]$ be the map which associates 
to $\tau =\sum_{k\geq 1} c_k \,2^{-k}$ the unique value $q$ which
satisfies the relation \eqref{eq:duality}. 
It is not difficult to check that $u$ is locally H\"older-continuous
and extends to a homeomorphism between $\Lambda'$ and the closure $\overline{\mathcal{U}}$
of $\mathcal{U}$ in $[1, 2]$, hence $\overline{\mathcal{U}}$ is also a Cantor set.
Moreover, one can prove that the H\"older exponent of $u$ restricted to the segments
$\Lambda_K$ (considered in Section \ref{HD}) is greater than the
Hausdorff dimension of $\Lambda_K$, thus $dim_H\ u(\Lambda_K)<1$ for all
$K \geq 1$ and $\overline{\mathcal{U}}$ has zero Lebesgue measure.


\end{remark}


\begin{thebibliography}{DEGH}

\bibitem[AB]{AB}
{\sc B Adamczewski, Y Bugeaud}, {\it On the complexity of algebraic numbers. II. Continued fractions},
Acta Math. {\bf 195} (2005), 1--20. 

\bibitem[A]{A}
{\sc J-P Allouche}, {\it Th\'eorie des Nombres et Automates}, Th\`ese d'\'Etat, 1983, Universit\'e Bordeaux I.

\bibitem[AC1]{AC1}
{\sc J-P Allouche, M Cosnard}, {\it It\'erations de fonctions unimodales et
suites engendr\'ees par automates}, C. R. Acad. Sci. Paris Sr. I Math. 
{\bf 296}  (1983),  no. 3, 159--162.

\bibitem[AC2]{AC2}
{\sc J-P Allouche, M Cosnard}, {\it Non-integer bases, iteration of continuous real maps, and an arithmetic self-similar set}, Acta Math. Hung. {\bf 91}  (2001), 325--332.

\bibitem[AC3]{AC3}
{\sc J-P Allouche, M Cosnard}, {\it The Komornik-Loreti constant is transcendental}, Amer. Math. Monthly {\bf 107} (2000), 448--449.
 
\bibitem[BI]{BI}
{\sc C Bonanno, S Isola}, {\it Orderings of the rationals and dynamical systems},  Coll. Math. {\bf 116} (2009), 165--189.

\bibitem[CT]{CT}
{\sc C Carminati, G Tiozzo}, {\it A canonical thickening of $\Q$ and the dynamics of continued fractions},  
to appear in Ergodic Theory and Dynamical Systems, Available on CJO 2011 doi:10.1017/S0143385711000447 

\bibitem[CMPT]{CMPT}
{\sc C Carminati, S Marmi, A Profeti, G Tiozzo}, {\it The entropy of $\alpha$-continued fractions: numerical results},  
Nonlinearity {\bf 23} (2010) 2429--2456.


\bibitem[Ca]{Ca}
{\sc J Cassaigne}
{\it Limit values of the recurrence quotient of Sturmian sequences},
   Theoret. Comput. Sci.  {\bf 218}  (1999),  no. 1, 3--12. 


\bibitem[dMvS]{dMvS}
{\sc W de Melo, S van Strien}, {\it One dimensional dynamics}, Springer-Verlag, Berlin, Heidelberg, 1993.

\bibitem[dV]{dV}
{\sc M de Vries},
{\it A property of algebraic univoque numbers},
Acta Math. Hungar. {\bf 119}, (2008), 57--62.

\bibitem[Do]{Do}
{\sc A Douady}, 
{\it Topological entropy of unimodal maps: monotonicity for quadratic polynomials},
in 
{\it Real and complex dynamical systems ({H}iller\o d, 1993)},
NATO Adv. Sci. Inst. Ser. C Math. Phys. Sci.
{\bf 464}, 65--87, Kluwer, Dordrecht, 1995.

\bibitem[EHJ]{EHJ} {\sc P Erd\"os, M Horv\'ath, I Jo\'o}, {\it On the uniqueness of the expansions $1=\sum q^{-n_i}$}, Acta Math. Hung. {\bf 58} (1991), 129--132.

\bibitem[EJK]{EJK} {\sc P Erd\"os, I Jo\'o, V Komornik}, {\it Characterization of the unique expansions 
$1=\sum q^{-n_i}$ and related problems} Bull. Soc. Math. France  {\bf 118} (1990),   no. 3, 377--390. 
\bibitem[F]{F} 
{\sc K Falconer}, {\it Fractal Geometry - Mathematical Foundations and Applications}, Wiley, 2003.

\bibitem[GS]{GS}
{\sc J Graczyk, G {\'S}wiatek}, {\it Generic hyperbolicity in the logistic family}, Ann. of Math. (2) {\bf 146} (1997), no. 1, 1--52.

\bibitem[He]{He} 
{\sc D Hensley}, {\it Continued fractions Cantor sets, Hausdorff dimension, and functional analysis}, J. Number Theory {\bf 40} (1992), 336--358. 

\bibitem[Is]{Is}
{\sc S Isola}, {\it  On a set of numbers arising in the dynamics of unimodal maps}, Far East J. Dyn. Syst. { \bf 6} (1) (2004), 79--96.
 
\bibitem[IP]{IP}
{\sc S Isola, A Politi}, {\it  Universal encoding for unimodal maps}, J. Stat. Phys. {\bf 61} (1990), 263--291.


\bibitem[KL1]{KL1}
{\sc V Komornik, P Loreti}, {\it Unique developments in non-integer bases},  Amer. Math. Monthly {\bf 105} (1998), 936--939.

\bibitem[KL2]{KL2}
{\sc V Komornik, P Loreti}, {\it On the topological structure of univoque sets},  J. Num. Th. {\bf 122} (2007), 157--183.

\bibitem[KSS]{KSS}
{\sc C Kraaikamp, T A Schmidt, W Steiner}, {\it Natural extensions and entropy of $\alpha$-continued fractions }, arXiv:1011.4283v1 [math.DS]

\bibitem[LM]{LM}
{\sc L Luzzi, S Marmi}, {\it On the entropy of Japanese continued fractions},
Discrete Contin. Dyn. Syst. {\bf 20} (2008), 673--711.

\bibitem[Ly]{Ly}
{\sc M Lyubich}, {\it Dynamics of quadratic polynomials. {I}, {II}}, Acta Math. {\bf 178} (1997), no. 2, 185--247, 247--297.

\bibitem[M]{M} {\sc K Mahler},
{\it Aritmetische Eigenschaften der L\"osungen einer Klasse von Funktionalgleichungen}, Math. Annalen {\bf
101} (1929), 342--366. Corrigendum {\bf 103} (1930), 532.

\bibitem[Ma]{Ma} {\sc N G Makarov},
{\it Conformal mapping and {H}ausdorff measures}, Ark. Mat. {\bf 25} (1987), no. 1, 41--89.


\bibitem[MT]{MT} {\sc J Milnor, W Thurston},
{\it On iterated maps of the interval}, Dynamical Systems (College Park, MD, 1986-87), 465--563, 
Lecture Notes in Math., 1342, Springer, Berlin, 1988.

\bibitem[Mo]{Mo} {\sc C G Moreira}, 
{\it Geometric properties of the Markov and Lagrange spectra}, 
preprint IMPA.

\bibitem[Na]{Na} {\sc H Nakada},
{\it Matrical theory for a class of continued fraction transformations and their natural extensions},
Tokyo J. Math. {\bf 4} (1981), 399--426. 

\bibitem[NN]{NN} 
{\sc H Nakada, R Natsui},
{\it The non-monotonicity of the entropy of $\alpha$-continued fraction transformations},
Nonlinearity {\bf 21} (2008), 1207--1225.

\bibitem[Sa]{Sa}
{\sc R Salem}, {\it On some singular monotone functions which are
strictly increasing},  TAMS {\bf 53} (1943), 427--439.



\bibitem[Th]{Th}
\textsc{W Thurston}, \textit{On the Geometry and Dynamics of Iterated Rational Maps}, in D Schleicher, N Selinger, editors, ``Complex dynamics'', 3--137, A K Peters, Wellesley, MA, 2009. 

\bibitem[Ti]{Ti}
\textsc{G Tiozzo},
\textit{The entropy of $\alpha$-continued fractions: analytical results}, 
arXiv:0912.2379v1 [math.DS].

\bibitem[Za]{Za}
\textsc{S Zakeri}, \textit{External Rays and the Real Slice of the Mandelbrot Set},
 Ergod. Th.  Dyn. Sys. \textbf{23} (2003) 637--660.



\end{thebibliography}
\end{document}